\documentclass{amsart}
\usepackage{amsmath}
  \usepackage{paralist}

\usepackage[pdfauthor={Oliver Butterley},
		pdftitle={An Alternative Approach to Generalised BV  and the Application to  Expanding Interval Maps},
		pdfsubject={Generalised BV  and Expanding Interval Maps},
		pdfkeywords={Generalised bounded variation, transfer operator, expanding map, spectral gap},
		colorlinks=true,
		citecolor=black,
		linkcolor=black]{hyperref}\usepackage{graphicx}

\newtheorem{theorem}{Theorem}[section]

\newtheorem*{main}{Main Theorem}
\newtheorem{lemma}[theorem]{Lemma}

\theoremstyle{definition}

\newcommand\cC{{\mathcal{C}}}
\newcommand\cJ{{\mathcal J}}

\newcommand\cL{{\mathcal L}}
\newcommand\cP{{\mathcal P}}

\newcommand\bC{{\mathbb C}}

\newcommand\bN{{\mathbb N}}

\newcommand\bR{{\mathbb R}}

\newcommand\bZ{{\mathbb Z}}

\newcommand{\Oo}{\Omega} 
\newcommand{\D}{D} 
\newcommand{\m}{\operatorname{\mathbf{m}}} 
\newcommand{\Lspace}[2]{\mathbf{L^{#1}}(#2)} 
\newcommand{\indicator}{\mathbf{1}} 
\newcommand{\TotalVar}[1]{\abs{ #1 }\!(\Oo)}   
\newcommand{\TotVar}[2]{\abs{ #1 }\!(#2)}   
\newcommand{\bvnorm}[1]{{\lVert{#1}\rVert}_{\mathbf{BV}(\Oo)}}
\newcommand{\norm}[1]{\left\lVert{#1}\right\rVert}
\newcommand{\abs}[1]{\left\lvert{#1}\right\rvert}
\newcommand{\LL}[3]{{\lVert{#2}\rVert}_{\mathbf{L^{#1}}(#3)}}

\newcommand{\gbv}[2]{\mathfrak{B}_{#1}(#2)}
\newcommand{\map}{f}
\newcommand{\BV}{\mathbf{BV}}
\newcommand{\cZ}{\mathcal{Z}}
\newcommand{\dmeasures}[1]{\mathfrak{D}(#1)}
\newcommand{\measures}{\mathfrak{M}(\Oo)}

\title[Generalised BV  and Expanding Interval Maps]{An Alternative Approach to Generalised BV  and the Application to  Expanding Interval Maps}
\author[Oliver Butterley]{Oliver Butterley}

\subjclass[2000]{Primary:  37D50;   Secondary:   37A05, 37E05, }
\keywords{Generalised bounded variation, transfer operator, expanding map, spectral gap}

\email{oliver.butterley@gmail.com}

\thanks{It is a pleasure to thank Carlangelo Liverani for many helpful discussions and comments. Research partially supported by the ERC Advanced Grant MALADY (246953). I am indebted  to Stefano Luzzatto for invaluable assistance during a period of many years.  I am grateful  to the library at ICTP where much of this work was done.}

\begin{document}
\maketitle

\begin{abstract}
We introduce a family of Banach spaces of measures, each containing the set of measures with density of bounded variation. These spaces are suitable for the study of weighted transfer operators of piecewise-smooth maps of the interval where the weighting used in the transfer operator is not better than piecewise H\"older continuous and the partition on which the map is continuous may possess a countable number of  elements. 
For such weighted transfer operators we give upper bounds for both the spectral radius and for the essential spectral radius.
\end{abstract}

\section{Introduction}\label{sec:intro}
An established and fruitful approach to the study of piecewise-smooth expanding  maps of the interval is to consider the push forward associated to the map which is a linear operator acting on the space of measures. Considered as a linear operator acting on densities the push forward is  called the transfer operator.
This operator has been studied acting on the space of functions of bounded variation  in great generality and has been shown to be quasi-compact from which many statistical properties follow by standard methods (see  \cite{Babook} and references within). 
Studying the transfer operator acting on the space of functions of  bounded variation requires that the inverse of the derivative of the map is a function of bounded variation. When less regularity exists a different space must be considered as the domain of the transfer operator. 
A question that arises naturally in dynamical systems, for example in the expanding map associated with Lorenz flows  \cite{LuzzMelPac2005}, is when the inverse of the derivative is only H\"older continuous or indeed piecewise H\"older continuous. Keller \cite{keller1985generalized} introduced function spaces which he called \emph{generalised bounded variation}. 
Acting on such spaces he  showed that the transfer operator associated to a piecewise expanding map is quasi-compact in the case where the map has  finite discontinuities and where the inverse of the derivative is H\"older continuous.

It would be desirable to extend the result of Keller in several different directions. We would like to deal with the case when there are a countable number of points of discontinuity and also study weighted transfer operators for some wide class of weights. 
An example of an application for both these possible extensions is seen when one is interested in studying flows. Commonly when studying a flow, or indeed a semiflow, one may consider some Poincar\'e section and so represent the flow as the combination of \emph{return map} and \emph{return time function} defined on that Poincar\'e section. Such a flow is called a suspension flow.  In certain situations one has the option to consider either a Poincar\'e section consisting of a finite number of connected components but with unbounded return time or alternatively choose a Poincar\'e section with a countable number of connected components but with the benefit that the return time is bounded. There is also an application for weighted transfer operators in the study of suspension flows. The knowledge of the spectral properties of weighted transfer operators gives information concerning the mixing rates of suspension flows via the use of twisted transfer operators (see for example \cite{Po}).

The space of generalised bounded variation of Keller  \cite{keller1985generalized} is based on the \emph{oscillation function} which measures the behaviour of the function close to a particular point and which are then integrated over the space to give the norm.
In contrast Thomine  \cite{MR2784627} uses Sobolev spaces with fractional order and recovers the results in the case where the derivative of the map is bounded and, as before, when the number of discontinuities are finite.
In the present work we take yet another  different approach in defining the Banach space which allows for significant simplifications of the calculations whilst being almost entirely self contained.
The result we obtain successfully extends the result of Keller \cite{keller1985generalized}  to the weighted transfer operators and to  the setting where there may be a countable number of discontinuities for the map and the weighting. We achieve this with only the rather weak additional assumption that there is an $\mathbf{L^{p}}$ bound on the derivative of the map.

\section{Generalised Bounded Variation}\label{sec:GBV}
Let $\Oo$ be the open unit interval. 
Let $\measures$ denote the set of complex measures on $\Oo$. 
This  is a Banach space with respect to the total variation norm $\TotalVar{\mu} = \sup\{ \abs{\mu(\eta)}: \eta \in \cC(\Oo), \abs{\eta}_{\infty} \leq 1  \}$ 
where $\cC(\Oo)$ denotes the set of complex-valued continuous functions with support contained within $\Oo$. 
We say that a measure  $\mu \in \measures$ is  differentiable (in the sense of measures) if there exists some measure  $D\mu \in \measures$ such that $\D\mu(\eta) = -\mu(\eta')$ for all $\eta\in \cC^{1}(\Oo)$.
Let $\dmeasures{\Oo}$ denote the set of all  differentiable measures. 
The set $\dmeasures{\Oo}$ is a Banach space when equipped with the norm $\bvnorm{\mu}:= \TotalVar{D\mu} +\TotalVar{\mu} $. 
This set corresponds to the set of measures with densities of {bounded variation}.
For any $\beta\in [0,1]$ and $\mu\in \measures$ let
\[
 \norm{\mu}_{\gbv{\beta}{\Oo}} := \inf \left\{ \sup_{k>0} \left(  k^{-\beta} \TotalVar{\mu_{k}-\mu} + k^{1-\beta} \norm{\mu_{k}}_{\BV(\Oo)} \right) \right\},
\]
where the infimum is taken over all 
{families of measures ${\{ \mu_{k} \}}_{k>0}\subset \dmeasures{\Oo}  $.}
We have  the understanding that this quantity may or may not be finite.
 {Note that the  families ${\{ \mu_{k} \}}_{k>0}$ are  parametrised by $k\in (0,\infty)$.  To understand the meaning of this observe that, for suitable approximating families,  $\mu_{k} \to \mu$ in $\TotalVar{\cdot}$ as $k\to0$ and $\norm{\mu_{k}}_{\BV(\Oo)}\to 0$ as $k\to \infty$.\footnote{There are several obvious alternatives to the definition of this norm, for example considering the analog but  changing to the set $k\in \{2^{n}: n\in \bZ\}$. See  \cite{Bergh:1976fk} for the discussion of the equivalence of such possibilities.}}
  We let $\gbv{\beta}{\Oo} := \{ \mu\in \measures : \norm{\mu}_{\gbv{\beta}{\Oo}}<\infty \}$. 
The quantity $ \norm{\cdot}_{\gbv{\beta}{\Oo}} $ has the required properties of  a seminorm on $\gbv{\beta}{\Oo} $.\footnote{The spaces defined here are equivalent to  the interpolation spaces defined using the real interpolation method \cite[Chapter 3]{Bergh:1976fk}. However in this article we take the point of view of not relying on any abstract theory and instead work explicitly and give  a self-contained account using the minimum required for the task at hand.}
   Note that the definition of $\norm{\cdot}_{\gbv{\beta}{\Oo}}$ 
   means that 
   \[
   \TotalVar{\mu} \leq \norm{\mu}_{\gbv{\beta}{\Oo}}
   \]
    (just write $\TotalVar{\mu} \leq \TotalVar{\mu-\mu_{k}} + \TotalVar{\mu_{k}}$ and take $k=1$) and so the newly defined quantity is a norm.
   Clearly $\gbv{\beta}{\Oo}$ is a non-empty set which at least contains all the measures with density  of bounded variation.
The case $\beta=1$ corresponds to the set of measures with densities of bounded variation (choose $\mu_{k}=\mu$ for all $k>0$) and the case $\beta=0$ corresponds to the set of all complex measures (choose $\mu_{k}=0$ for all $k>0$). 
Note that $\gbv{\beta}{\Oo}$ is a vector subspace of $\measures$. Furthermore the norm $\norm{\cdot}_{\gbv{\beta}{\Oo}}$ is lower semicontinuous: Whenever ${\{\mu_{n}\}}_{n=1}^{\infty} \subset \gbv{\beta}{\Oo}$ and $\TotalVar{\mu_{n}-\mu} \to 0$ as $n\to \infty$ then $\norm{\mu}_{\gbv{\beta}{\Oo}} \leq \liminf_{n\to \infty} \norm{\mu_{n}}_{\gbv{\beta}{\Oo}}$. This means that $(\gbv{\beta}{\Oo}, \norm{\cdot}_{\gbv{\beta}{\Oo}})$ is a Banach space. Also note that $ \gbv{\beta'}{\Oo}   \subset \gbv{\beta}{\Oo} $ whenever $\beta' <\beta$. 
These spaces are convenient for several reasons. In particular that for all $\beta \in (0,1]$ the embedding 
\begin{equation}\label{eq:compact}
\gbv{\beta}{\Oo} \hookrightarrow \measures \quad \text{ is compact}.
\end{equation}
This is a property which is inherited from the compact embedding of $\gbv{1}{\Oo} = \dmeasures{\Oo}$ into $\measures$.   The proof of  \eqref{eq:compact} 
  is not lengthy and, for the sake of completeness, is included in Section~\ref{sec:basicfacts}.
Additionally  $\gbv{\beta}{\Oo}$ is an exact interpolation space \cite{Bergh:1976fk} although we will not use this directly in this article.
This means that if  $\cL : \measures \to \measures$ is both continuous as an operator from $\measures$ to itself and  also from $ \dmeasures{\Oo}$ to itself, then $\cL :\gbv{\beta}{\Oo} \to \gbv{\beta}{\Oo}  $ is continuous  and furthermore\footnote{We emphasise  the space on which the operator is being considered: By $\smash{\norm{\cL}_{\measures}}   $ we  mean the operator norm of  $\cL :\measures \to \measures$ and similarly for $ \gbv{\beta}{\Oo}$ and $\BV(\Oo)$. }  
\begin{equation}\label{eq:exact}
 \norm{\cL}_{\gbv{\beta}{\Oo}} \leq \norm{\cL}_{\measures}^{1-\beta} \norm{\cL}_{\BV(\Oo)}^{\beta}.
 \end{equation}
This approach has the benefit, compared to Keller's generalised bounded variation, of being extremely simple to define and allowing instant access to the approximation sequences. An advantage demonstrated by the extreme brevity of the present work.

\section{Piecewise Expanding Maps}
{As before $\Oo$ denotes  the open unit interval.}
{We suppose that $\cZ \subset \Oo$ is a closed set of zero Lebesgue measure such that $\Oo\setminus\cZ$ is the countable (or indeed finite) union of open intervals.}
 {We suppose that} we are given  $\map \in \cC^{1}(\Oo\setminus \cZ, \Oo)$ which we call the \emph{map} and $\xi: \Oo\setminus \cZ \to \bC$ which we call the \emph{weighting}. 
We require that the map is expanding, i.e. that $\inf\{ \map'(x):{x\in \Oo\setminus\cZ}\} >1$.
Let $\{ \omega_{j}\}_{j\in \cJ}$ 
denote the set of connected components of $\Oo\setminus \cZ$. 
We further suppose that there exists $\alpha \in (0,1)$ such that
\begin{equation}\label{eq:assumption2}
\xi \in \cC^{\alpha}( \Oo\setminus \cZ, \bC) \quad \text{and} \quad \sum_{j\in \cJ} \LL{\infty}{\xi}{\omega_{j}} <\infty
\end{equation}
Let us clarify what we mean by the first statement in the above  since the domain $\Oo\setminus \cZ$ is not connnected. We mean that there exists  $C<\infty$ such that ${\abs{\xi(x)-\xi(y)}}\leq C {\abs{x-y}^{\alpha}}$
for all $x,y \in \omega_{j}$, for all $j\in \cJ$. I.e. that the H\"older coefficient can be chosen  uniformly for all $\omega_{j}$. 
We suppose
\begin{equation}\label{eq:assumption1}
\LL{p}{f'}{\Oo} <\infty
\quad \text{and} \quad
\LL{\infty}{\xi \cdot \map'}{\Oo} <\infty
\end{equation}
for some $p> \frac{1 }{ \alpha}$.
For any $n\in \bN$ let $\xi^{{(n)}}:= \prod_{i=0}^{n-1} \xi\circ \map^{i}$.
We finally suppose that
\begin{equation}\label{eq:assumption3}\begin{split}
\lambda_{1}&:= \lim_{n\to \infty} \LL{\infty}{\xi^{(n)}}{\Oo}^{\frac{1}{n}} < { \infty},\\ 
\lambda_{2}&:= \lim_{n\to \infty} \LL{\infty}{\xi^{(n)} \cdot (\map^{n})'}{\Oo}^{\frac{1}{n}}<\infty.
\end{split}
\end{equation}
The limits in the above exist by submultiplicativity. 
For each map $\map$ and weighting $\xi$ as introduced above we define the weighted transfer operator $\cL_{\xi, \map} : \measures\to\measures$ by
\[
\cL_{\xi,\map}\mu(\eta) := \mu(\xi \cdot f' \cdot \eta \circ f) \quad \text{for all $\eta \in \cC(\Oo)$}.
\]
The  transfer operator\footnote{Written in terms of densities of the measures, as is more common, the above defined transfer operator is given by
$\cL_{\xi,\map}h = \sum_{j}(\xi \cdot h) \circ f_{j}^{-1} \cdot \indicator_{f(\omega_{j})}$
where $f_{j}:= \left.f\right|_{\omega_{j}}$. }
 which corresponds to the push forward 
$ \map_{*}\mu(\eta) := \mu(\eta\circ \map)$ is given by the choice $\xi = \frac{1}{f'}$. In this case  $\lambda_{2}=1$. 
We know that {${ {\norm{\cL_{\xi,\map}}_{\measures} }} \leq   \LL{\infty}{\xi \cdot f'}{\Oo}$}
but beyond this these linear operators do not have good spectral properties acting on $\measures$ but they do have good spectral properties as operators acting on $\gbv{\beta}{\Oo}$ as we will see subsequently. 
\begin{main}
Suppose  the map $\map\in \cC^{1}(\Oo\setminus \cZ, \Oo)$ and the weighting $\xi:\Oo \to \bC$ are as introduced above satisfying \eqref{eq:assumption2}, \eqref{eq:assumption1},   and \eqref{eq:assumption3}. Let $
\beta = \alpha $. Then  
$\cL_{\xi,\map} : \gbv{\beta}{\Oo} \to \gbv{\beta}{\Oo} $ {has} 
 spectral radius not greater than $\lambda_{2}$ and  
  essential spectral radius not greater than $\lambda_{1}^{\beta} \lambda_{2}^{1-\beta}$.
\end{main}
The above theorem is proven  in Section~\ref{sec:proof}, using various results concerning $\gbv{\beta}{\Oo}$ which are proven in Section~\ref{sec:basicfacts}. First we make some comments.  
Note that the above theorem only gives upper bounds on the spectral radius and essential spectral radius and therefore doesn't  prove that the operators are quasi-compact. For this a lower bound for the spectral radius would be required. Such an estimate is, in general,  not trivial to prove.
Setting $\xi = \frac{1}{\map'}$ we recover the results of Keller \cite{keller1985generalized} but we require the mild additional condition $\LL{p}{\map'}{\Oo} <\infty$ whilst he does not. This is significantly better than the condition of  $\map'$ being bounded 
  which is  required by the work of Thomine   \cite{MR2784627}. 
The issue of allowing unbounded expansion near the points of discontinuity is key in many dynamical systems. For example in the Poincar\'e return maps for singular hyperbolic flows and in billiard maps due to grazing collisions. 
To conclude we note that this work successfully extends the results to the case with countable points of discontinuity and to a large class of weighted transfer operators which has the benefits discussed  in Section~\ref{sec:intro}.

\section{Basic Properties of $\gbv{\beta}{\Oo}$}\label{sec:basicfacts}
Here we prove some properties concerning $\gbv{\beta}{\Oo}$ which we will require in the next section.
\begin{lemma}\label{lem:compact}
For all $\beta \in (0,1]$ the embedding $\gbv{\beta}{\Oo} \hookrightarrow \measures$ is compact.
\end{lemma}
\begin{proof}
Fix a sequence ${\{\mu_{n}\}}_{n=1}^{\infty}\subset \gbv{\beta}{\Oo}$ such that $\norm{\mu_{n}}_{\gbv{\beta}{\Oo}}\leq \frac{1}{2}$ for all $n$.
By the definition of  $\norm{\cdot}_{\gbv{\beta}{\Oo}}$ for each $n$ there exists a sequence ${\{{\mu_{n,m}}\}}_{m =1}^{\infty} \subset \dmeasures{\Oo}$ such that (here we choose $k=k(m)=2^{-m}$)
\begin{equation}\label{eq:jo}
\TotalVar{  \mu_{n,m} - \mu_{n} } \leq  2^{-m \beta} \quad \text{and} \quad 
\TotalVar{  \D\mu_{n,m} } \leq  2^{m(1-\beta)}
\quad \text{for all $m\in \bN$}.
\end{equation}Fixing for the moment $m=1$ we consider the sequence ${\{\mu_{n, m}\}}_{n=1}^{\infty}$. This is a bounded subset of the space of measures with density of bounded variation by the second estimate of~\eqref{eq:jo} and so there exists a subsequence of indexes  ${\{n_{i_{m}}\}}_{i_{m}=1}^{\infty}$ such that the sequence ${\{\mu_{n_{i_{m}},m}\}}_{i_{m}=1}^{\infty}$  converges in $\TotalVar{\cdot}$.
Next we repeat for $m=2$ and 
we proceed in such a manner for all $m\in \bN$ and obtain  the diagonal sequence  ${\{\mu_{n_{i_{m}},m}\}}_{m=1}^{\infty}$ which also  converges in $\TotalVar{\cdot}$. Using this and the first estimate from \eqref{eq:jo} we have shown that 
  ${\{\mu_{n_{i_{m}}}\}}_{{m}=1}^{\infty}$ converges in $\TotalVar{\cdot}$.
\end{proof}

\begin{lemma}\label{lem:muk}
Suppose  $\mu\in \gbv{\beta}{\Oo}$, $M>\norm{\mu}_{\gbv{\beta}{\Oo}}$  and
 that $\{\mu_{k}\}_{k>0} \subset \dmeasures{\Oo} $ satisfies
\[
  k^{-\beta} \TotalVar{\mu_{k}-\mu} + k^{1-\beta} \norm{\mu_{k}}_{\BV(\Oo)} 
\leq M
\]
for all $k>0$.
Then $\norm{\mu_{k}}_{\gbv{\beta}{\Oo}} \leq 2 M$ for all $k>0$.
\end{lemma}
\begin{proof}
Fix $k>0$. Define $\nu_{j}\in \dmeasures{\Oo}$ for all $j>0$ by
\[
\nu_{j} := \begin{cases}
\mu_{j} & \text{if $j\geq k$}\\
\mu_{k} & \text{if $j< k$}.
\end{cases}
\]
We use this as an approximating sequence to estimate $\norm{\mu_{k}}_{\gbv{\beta}{\Oo}}$. First for $j\geq k$ we have
\[
\begin{split}
 j^{-\beta} \TotalVar{\nu_{j}-\mu_{k}} + j^{1-\beta} \norm{\nu_{j}}_{\BV(\Oo)}
 &= j^{-\beta} \TotalVar{\mu_{j}-\mu_{k}} + j^{1-\beta} \norm{\mu_{j}}_{\BV(\Oo)}\\
 &\leq  j^{-\beta} \TotalVar{\mu_{j}-\mu}   + j^{1-\beta} \norm{\mu_{j}}_{\BV(\Oo)}\\
 &\ \ \ + k^{-\beta} \TotalVar{\mu-\mu_{k}}\\
 & \leq 2 M.
\end{split}
\]
Additionally we have $j^{-\beta} \TotalVar{\nu_{j}-\mu_{k}} + j^{1-\beta} \norm{\nu_{j}}_{\BV(\Oo)}
= j^{1-\beta} \norm{\nu_{k}}_{\BV(\Oo)}
\leq M$
in  the case $j< k$.
\end{proof}

\begin{lemma}
\label{lem:Lpbound}
Suppose $\beta \in (0,1)$, $p>\frac{1}{\beta}$. There exists $C<\infty$ such that
\[
\abs{\mu (\eta) } \leq C \LL{p}{\eta}{\Oo} \norm{\mu}_{\gbv{\beta}{\Oo}}
\]
for all $\eta \in \Lspace{p}{\Oo}$ and $\mu \in \gbv{\beta}{\Oo}$.
\end{lemma}
{The above lemma has the interesting consequence that if $\mu \in    \gbv{\beta}{\Oo}$ then $\mu$ has density in $\Lspace{q}{\Oo}$ for all $q<\frac{\beta^{-1}}{\beta^{-1}-1}$.  }
\begin{proof}[Proof of Lemma~\ref{lem:Lpbound}]
Fix  $\eta \in \Lspace{p}{\Oo}$ and $\mu\in \gbv{\beta}{\Oo}$. For $n\in \bN$ let $a_{n} :=  2^{n} \LL{p}{\eta}{\Oo}$ and hence let
\[
\begin{split}
A_{0} &:= \{ x\in \Oo : \abs{\eta}(x) \leq a_{0} \},\\
A_{n} &:= \{ x\in \Oo : a_{n-1} < \abs{\eta}(x) \leq a_{n} \},
\end{split}
\]
for all $n\in \{1,2,\ldots\} $.
Note that these sets are disjoint and that $\bigcup_{n=0}^{\infty} A_{n} = \Oo$. 
Further note that\footnote{We use the notation $\m$ to denote Lebesgue measure on $\Oo$.}  $\LL{p}{\eta}{\Oo} \geq ( \m(A_{n}) 2^{n p} \LL{p}{\eta}{\Oo}^{p})^{\frac{1}{p}}$ and so
\[
\m(A_{n}) \leq 2^{-n p} \quad \quad \text{for all $n\in \bN$}.
\]
We estimate $\abs{\mu( \eta \cdot \indicator_{A_{n}} )} \leq { a_{n}} \abs{\mu (\indicator_{A_{n}})}$. 
Since $\mu\in \gbv{\beta}{\Oo}$ for each $M>\norm{\mu}_{\gbv{\beta}{\Oo}}$ there exists 
$\{\mu_{k}\}_{k>0} $ such that,  for all $k>0$ 
\begin{equation}
k^{-\beta} \TotVar{\mu - \mu_{k}}{\Oo} +  k^{1-\beta} \bvnorm{\mu_{k}}
\leq M.
\end{equation}
Fix for the moment $n\in \bN$. Let $k= k(n) = 2^{-p n}$. We have\footnote{\label{ft:BVL1}We have
$\abs{\mu(\eta)} \leq  2 \norm{\mu}_{\BV(\Oo)} \LL{1}{\eta}{\Oo}$
for all $\mu\in \dmeasures{\Oo}$, $\eta\in \Lspace{1}{\Oo}$.
It suffices to prove this for $\eta\in \cC(\Oo)$ by Lusin's Theorem. Let
$\phi(x):= \int_{0}^{x}\eta(y) \ dy - x  \int_{\Oo}\eta(y) \ dy$ where $\Oo =(0,1)$.
This means that $\phi(0) = \phi(1) = 0$ and that $\phi'(x) = \eta(x) - \int_{\Oo} \eta(y) \ dy$. Consequently $\LL{\infty}{\phi}{\Oo} \leq 2 \LL{1}{\eta}{\Oo}$ and 
$\mu(\eta) = - D\mu(\phi) +  \mu(1) \int_{\Oo} \eta(y) \ dy$.}
\[
\begin{split}
\abs{\mu (\indicator_{A_{n}})} 
&\leq \TotalVar{\mu - \mu_{k}} + \abs{\mu_{k} (\indicator_{A_{n}})}\\
& \leq ( k^{\beta} + 2^{-n p} k^{-(1-\beta)})M \leq 2^{-n \beta p} (2 M).
\end{split}
\]
This means that $2^{n} \abs{\mu (\indicator_{A_{n}})}  \leq 2^{-n( \beta p - 1)} (2 M)$ and we recall that $\beta p >1 $ by assumption and so this quantity is summable over $n$.
This means that there exists   $C<\infty$, dependent only on $p-\frac{1}{\beta}$, such that
\[
\abs{\mu(\eta)} \leq  \sum_{n=0}^{\infty} \abs{\mu(\eta \cdot \indicator_{A_{n}})} \leq C M  \LL{p}{\eta}{\Oo}.\qedhere
\]
\end{proof}

\section{Proof of the Main Theorem} 
\label{sec:proof}
At the moment we consider $f:\Oo\setminus \cZ \to\Oo$ and $\xi:\Oo\setminus\cZ \to \bC$ to be  fixed and satisfying 
\eqref{eq:assumption2}, \eqref{eq:assumption1},   and \eqref{eq:assumption3}.
Recall that $\xi$ is assumed to be $\alpha$-H\"older on each $\omega_{j}$ with H\"older coefficient uniform for all $j$.
We require a smoothed version of $\xi$.
In order to construct this we will use convolution with a smooth mollifier: Fix $\rho\in \cC^{1}(\bR, [0,1])$ with support contained within $(-1,1)$ and which satisfies $\int_{-1}^{1} \rho(x) \ dx =1$ and $\sup_{x\in \bR} \abs{\rho'(x)} \leq 2$. 
For all $\epsilon>0$ let $\rho_{\epsilon}(x):=  \epsilon^{-1}\rho(\epsilon^{-1} x)$. 
Note that $\rho_{\epsilon}$ has support contained within $(-\epsilon,\epsilon)$.
Fix for the moment $j \in \cJ$.
Let $\tilde \xi :\bR \to \bC $ denote the continuous function which is equal to ${\xi}$ on $\omega_{j}$ and constant elsewhere.
For each $\epsilon>0$ let $\xi_{\epsilon}: \omega_{j}\to \bC$ be defined as $\xi_{\epsilon}:=\rho_{\epsilon}*\tilde\xi$. Note that $\xi_{\epsilon}\in \cC^{1}(\omega_{j},\bC)$ and $\LL{\infty}{\xi_{\epsilon}}{\omega_{j}}  \leq \LL{\infty}{\xi }{\omega_{j}} $.
By performing this construction for each $j$ we define $\xi_{\epsilon}: \Oo\setminus \cZ \to \bC$.
We have the following estimates, a standard property\footnote{It suffices to note that
$(\xi_{\epsilon}-\xi)(x) = \int \rho_{\epsilon}(x-y)[\xi(y)-\xi(x)] \ dy$ and  that $\xi_{\epsilon}'(x) = \int \rho_{\epsilon}'(x-y)[\xi(y)-\xi(x)] \ dy$.} for H\"older continuous functions:
There exists $C_{\xi}<\infty$  such that  for all  $\epsilon>0$  
\begin{equation}\label{eq:smoothxi}
\LL{\infty}{ \xi_{\epsilon}-\xi}{\Oo} \leq C_{\xi} \epsilon^{\alpha} 
\quad \text{and} \quad
\LL{\infty}{ \xi_{\epsilon}' }{\Oo} \leq C_{\xi}  \epsilon^{-(1-\alpha)}.
\end{equation}For all $\epsilon>0$ we let $\cP_{\epsilon}:= \cL_{\xi_{\epsilon},\map}$. I.e. for each $\mu \in \measures$
the operator is defined as
$\cP_{\epsilon}\mu(\eta) = \mu( \xi_{\epsilon}\cdot \map' \cdot \eta \circ \map)$ for all $\eta \in \cC(\Oo)$.

\begin{lemma}\label{lem:estBV}
There exists $C<\infty$ such that, for all $\epsilon>0$ and $\mu\in \dmeasures{\Oo}$
\[
\norm{ \cP_{\epsilon} \mu  }_{\BV(\Oo)} \leq  { 6}  \LL{\infty}{\xi}{\Oo} \norm{\mu}_{\BV(\Oo)}  + C \epsilon^{-(1-\alpha)} \TotVar{\mu}{\Oo}.
\]
\end{lemma}
\begin{proof}
Since $\sum_{j=0}^{\infty}   \LL{\infty}{  \xi }{\omega_{j}} <\infty$ by assumption \eqref{eq:assumption2},  we may choose some $j_{0}<\infty$ sufficiently large  such that (or in the case that the set $\cZ$ is  finite this step may of course be omitted)
\begin{equation}\label{eq:joel}
\sum_{j=j_{0}+1}^{\infty}   \LL{\infty}{  \xi }{\omega_{j}} \leq  \LL{\infty}{  \xi }{\Oo}. 
\end{equation}Let\footnote{We use the notation  $\abs{\omega}$ to denote the length of any interval $\omega$.} $C_{j_{0}}:= \sup\{  \abs{\omega_{j}}^{-1}: j\in \{1,2,\ldots, j_{0}\}\} < \infty$. 
Fix now $\mu\in \dmeasures{\Oo}$. For all $\eta \in \cC^{1}(\Oo,\bC)$  we have  
 $
{ \cP_{\epsilon}\mu(\eta')}
 =  \mu( [\eta\circ \map  \cdot \xi_{\epsilon}  ]'  )
  -  \mu( \eta\circ \map  \cdot \xi_{\epsilon}'  )$.
  Note that $( \eta \circ \map  \cdot \xi_{\epsilon} ) \in \cC^{1}(\omega_{j}, \bC) $ for each $j$ but may be discontinuous on $\cZ$. 
We let $\psi_{\eta,\epsilon}:\Oo\to\bC$ denote the function which is affine on each $\omega_{j}=(a_{j},b_{j})$ and is such that $( \eta\circ \map  \cdot \xi_{\epsilon} -  \psi_{\eta,\epsilon} )(x) \to 0$ as $x\nearrow b_{j}$ and as $x\searrow a_{j}$ for each $j$. 
We can now write
 \begin{equation}\label{eq:gogo}
\cP_{\epsilon}\mu(\eta')
 =  \mu \left( (\eta\circ \map \cdot \xi_{\epsilon}  - \psi_{\eta,\epsilon} )'  \right)
+  \mu \left(  \psi_{\eta,\epsilon}'  \right)
  -  \mu  \left( \eta \circ\map  \cdot \xi_{\epsilon}' \right).
 \end{equation}
 Since  $ [\eta\circ \map  \cdot \xi_{\epsilon} - \psi_{\eta,\epsilon}] \in \cC(\Oo)$ and  $\LL{\infty}{\psi_{\eta,\epsilon}}{\omega_{j}} \leq \LL{\infty}{  \xi_{\epsilon} }{\omega_{j}}$ the first term may be estimated as $\lvert{ \mu( [\eta\circ \map  \cdot \xi_{\epsilon} - \psi_{\eta,\epsilon}]'  )}\rvert \leq 2 \LL{\infty}{  \xi }{\Oo}  \TotVar{ D\mu}{\Oo}$.
We turn our attention to the second term.
Note that $\LL{\infty}{\psi_{\eta,\epsilon}'}{\omega_{j}} \leq 2 \LL{\infty}{  \xi }{\omega_{j}} \abs{\omega_{j}}^{-1}$ which means that $\LL{1}{\psi_{\eta,\epsilon}'}{\omega_{j}} \leq 2 \LL{\infty}{  \xi }{\omega_{j}}$.
Hence, by \eqref{eq:joel} and reusing  the comment of Footnote~\ref{ft:BVL1},  we have that 
\[
\lvert{  \mu(  \psi_{\eta,\epsilon}'  ) }\rvert  \leq   { 4} \LL{\infty}{  \xi }{\Oo} \norm{\mu}_{\BV(\Oo)} + 2 C_{j_{0}}  \LL{\infty}{  \xi }{\Oo} \TotVar{\mu}{\Oo}.
\] 
For the final term of \eqref{eq:gogo} we have
$\lvert{\mu ( \eta\circ \map  \cdot \xi_{\epsilon}' )}\rvert \leq C_{\xi} \epsilon^{-(1-\alpha)}  \TotVar{\mu}{\Oo}$ where we used the estimate for $\LL{\infty}{\xi_{\epsilon}'}{\Oo}$ from \eqref{eq:smoothxi}.
Summing these above estimates for the three terms of \eqref{eq:gogo} we have shown that\footnote{ {We do not claim that the constants which appear here are optimal. This is not required for this argument. A similar comment applies to the constants in subsequent calculations. They will have no impact on the final estimate of the essential spectral radius.}}
\[
\TotVar{\D \cP_{\epsilon}\mu}{\Oo} \leq   { 6}  \LL{\infty}{  \xi }{\Oo} \norm{\mu}_{\BV(\Oo)} + \left( 2 C_{j_{0}}  \LL{\infty}{  \xi }{\Oo} 
+ C_{\xi} 2^{(1-\alpha)\epsilon} \right)  \TotVar{\mu}{\Oo}.
\]
Furthermore we have the  estimate $\TotalVar{\cP_{\epsilon}\mu} \leq \LL{\infty}{\xi\cdot\map'}{\Oo} \TotalVar{\mu}$.
\end{proof}

\begin{lemma}\label{lem:LY}
There exists $C<\infty$ such that, for all $\mu \in \gbv{\beta}{\Oo}$ 
\[
\norm{\cL_{\xi,f}\mu}_{\gbv{\beta}{\Oo}} \leq  10 \LL{\infty}{\xi}{\Oo}^{\beta} \LL{\infty}{\xi \cdot f'}{\Oo}^{1-\beta} \norm{\mu}_{\gbv{\beta}{\Oo}} 
+ C \TotVar{\mu}{\Oo}.
\]
\end{lemma}

\begin{proof}
Since $\mu\in { \gbv{\beta}{\Oo}}$ for each $M>\norm{\mu}_{\gbv{\beta}{\Oo}}$ there exists 
$\{\mu_{k}\}_{k>0} $ such that,  for all $k>0$ 
\begin{equation}\label{eq:M}
k^{-\beta} \TotVar{\mu - \mu_{k}}{\Oo} +  k^{1-\beta} \bvnorm{\mu_{k}}
\leq M.
\end{equation}
We fix some $\ell_{0}>0$ and $\epsilon_{0}>0$ (which will be chosen below). Let $k(\ell):=    \LL{\infty}{\xi}{\Oo} \LL{\infty}{\xi \cdot f'}{\Oo}^{-1}   \ell $ and let $\epsilon(\ell):= \epsilon_{0} \ell $ for all $\ell>0$.   Now let
 \[
 \nu_{\ell}:= \begin{cases}
 \cP_{\epsilon(\ell)}\mu_{k(\ell)}    &  \text{ if $\ell \in (0, \ell_{0})$}\\
 0 & \text{ if $\ell \geq \ell_{0}$}.
 \end{cases}
 \]
 For $\ell\geq \ell_{0}$ we have immediately that 
 \begin{equation}\label{eq:yeah} 
 \begin{split}
 \ell^{-\beta} \TotalVar{\cL_{\xi,f}\mu - \nu_{\ell}} + \ell^{1-\beta} \norm{\nu_{\ell}}_{\BV({\Oo})} 
 & = \ell^{-\beta}    \TotalVar{\cL_{\xi,f}\mu}    \\
 & \leq \ell_{0}^{-\beta}   
  \LL{\infty}{\xi \cdot f'}{\Oo}  \TotalVar{\mu}.
  \end{split}
 \end{equation} 
  Now we consider  $\ell \in (0, \ell_{0})$.
  First we  estimate $\ell^{-\beta} \TotalVar{\cL_{\xi,f}\mu - \nu_{\ell}} $.  Note that
  \[
  \begin{split}
  \left( \cL_{\xi,f}\mu - \nu_{\ell}\right) (\eta) &= \mu(\xi \cdot \map' \cdot {\eta \circ \map}) - \mu_{k}(\xi_{\epsilon}\cdot \map'\cdot \eta\circ \map)\\
  &= (\mu -\mu_{k})(\xi\cdot\map'\cdot\eta\circ\map) + \mu_{k}([\xi-\xi_{\epsilon}]\cdot \map'\cdot\eta\circ\map).
  \end{split}
  \]
  Using the above calculation and Lemma~\ref{lem:Lpbound} we have that for all $\ell \in (0,\ell_{0})$
 \begin{equation*}
 \begin{split}
\TotalVar{\cL_{\xi,f}\mu - \nu_{\ell}} 
  &\leq    \LL{\infty}{\xi \cdot f'}{\Oo}  \TotalVar{\mu - \mu_{k(\ell)}} \\
& \ \ \   + C \LL{\infty}{\xi-\xi_{\epsilon(\ell)}}{\Oo} \LL{p}{\map'}{\Oo} \norm{\mu_{k(\ell)}}_{\gbv{\beta}{\Oo}}.
 \end{split}
 \end{equation*} 
 By definition of $k(\ell)$ we have $\ell^{-\beta} = k(\ell)^{-\beta}    \LL{\infty}{\xi}{\Oo}^{\beta}  \LL{\infty}{\xi \cdot f'}{\Oo}^{-\beta}  $. We also recall Lemma~\ref{lem:muk}, the  definition of $\epsilon(\ell)$ and \eqref{eq:smoothxi}.
 We have
 \begin{equation}\label{eq:wow}
 \begin{split}
 \ell^{-\beta} \TotalVar{\cL_{\xi,f}\mu - \nu_{\ell}} 
  &\leq  \LL{\infty}{\xi }{\Oo}^{\beta}   \LL{\infty}{\xi \cdot f'}{\Oo}^{1-\beta} M \left( 1+ 2C_{\xi} \epsilon_{0}^{\alpha} \LL{p}{\map'}{\Oo}  \right)\\
  &\leq  2 C \LL{\infty}{\xi }{\Oo}^{\beta}   \LL{\infty}{\xi \cdot f'}{\Oo}^{1-\beta} M,
 \end{split}
 \end{equation} 
 where we now choosen $\epsilon_{0}>0$ dependent only on $\xi$ and $\map$ such that the last line of the above holds.

Now we  estimate $\ell^{1-\beta}\norm{\nu_{\ell}}_{\BV({\Oo})}$. Using the estimate of Lemma~\ref{lem:estBV} 
 we have (increasing $C<\infty$ if required) that
 \[
\norm{ \nu_{\ell} }_{\BV(\Oo)} \leq 4 \LL{\infty}{\xi}{\Oo} \norm{\mu_{k(\ell)}}_{\BV(\Oo)}  +  C \epsilon(\ell)^{-(1-\alpha)} \TotVar{\mu_{k(\ell)}}{\Oo}.
 \]
 For the first term note that
$ \ell^{1-\beta} \LL{\infty}{\xi}{\Oo}  =  \LL{\infty}{\xi}{\Oo}^{\beta}  \LL{\infty}{\xi \cdot f'}{\Oo}^{1-\beta}    k(\ell)^{1-\beta} $  and so   $\ell^{1-\beta} \LL{\infty}{\xi}{\Oo} \norm{\smash{\mu_{k(\ell)}}}_{\BV(\Oo)} \leq     \LL{\infty}{\xi}{\Oo}^{\beta}  \LL{\infty}{\xi \cdot f'}{\Oo}^{1-\beta}    M$. 
 Concerning the second term note that $\TotalVar{\mu_{k}} \leq \TotalVar{\mu_{k}-\mu} + \TotalVar{\mu}$ and so $\TotalVar{\mu_{k}} \leq k^{\beta} M + \TotalVar{\mu}$ and hence  \[\begin{split}
 C \epsilon(\ell)^{-(1-\alpha)}\ell^{1-\beta}\TotalVar{\mu_{k(\ell)}} 
 & \leq C \epsilon_{0}^{-(1-\alpha)}\TotalVar{\mu_{k(\ell)}}\\
 &\leq     C \epsilon_{0}^{-(1-\alpha)} \left(  k(\ell)^{\beta}   M + \TotalVar{\mu}\right).
 \end{split}
 \]
 We now choose, as promised above, $\ell_{0}$ sufficiently small so that
  that
  \begin{equation}\label{eq:oh}
\ell^{1-\beta} \norm{ \nu_{\ell} }_{\BV(\Oo)} \leq 5 \LL{\infty}{\xi}{\Oo}^{\beta}  \LL{\infty}{\xi \cdot f'}{\Oo}^{1-\beta}   M
 + \tilde C \TotVar{\mu}{\Oo},
 \end{equation} for all $\ell \in (0,\ell_{0})$ for some $\tilde C<\infty$.
 By the estimates of \eqref{eq:yeah}, \eqref{eq:wow} and \eqref{eq:oh} we have shown that there exists $ C<\infty$ (dependent only on $f$, $\xi$ and $(p-\frac{1}{\beta})$) such that 
\[
 \ell^{-\beta} \TotalVar{\cL_{\xi,f} - \nu_{\ell}} + \ell^{1-\beta} \norm{\nu_{\ell}}_{\BV({\Oo})}
 \leq 10  \LL{\infty}{\xi}{\Oo}^{\beta}  \LL{\infty}{\xi \cdot f'}{\Oo}^{1-\beta}   M
 +  C  \TotVar{\mu}{\Oo},
\]
  for all $\ell >0$. That this holds for all $M>\norm{\mu}_{\gbv{\beta}{\Oo}}$ completes the proof.
\end{proof}

\begin{lemma}\label{lem:LY2}
For each $n\in \bN$ there exists\footnote{By iterating this estimate one could easily remove the dependence of $C_{n}$ on $n$ but this serves no benefit is the present argument.} $C_{n}<\infty$ such that, for all $\mu \in \gbv{\beta}{\Oo}$,  
\[
\norm{\cL_{\xi,f}^{n}\mu}_{\gbv{\beta}{\Oo}} \leq  10 \LL{\infty}{\xi^{(n)}}{\Oo}^{\beta} \LL{\infty}{\xi^{(n)} \cdot (f^{n})'}{\Oo}^{1-\beta} \norm{\mu}_{\gbv{\beta}{\Oo}} 
+ C_{n} \TotVar{\mu}{\Oo}.
\]
\end{lemma}
\begin{proof}
The weighted transfer operator $\cL_{\xi,\map}$ is the one associated  to the map $f$ and weight $\xi$.  We now wish to consider $\cL_{\xi,\map}^{n}$. This however is equal to the transfer operator $\cL_{\tilde \xi, \tilde \map}$, the one associated  to the map $\smash{\tilde f}:= f^{n}$ and weight $\smash{\tilde\xi}:= \xi^{(n)}$. Assumptions \eqref{eq:assumption2} and \eqref{eq:assumption2} continue to hold since the map $\map$ is expanding.
This means that this lemma is a direct consequence of  Lemma~\ref{lem:LY2}.
\end{proof}

We are now in a position to complete the proof of the Main Theorem 
by using the above estimate. 
That $\cL_{\xi,f}: \gbv{\beta}{\Oo} \to \gbv{\beta}{\Oo}$  is continuous is immediate from Lemma~\ref{lem:LY2}. 
 To bound  the essential spectral radius we follow Hennion's argument \cite{He}.
Let 
 \[
 B_{n}:= \{\cL_{\xi,f}^{n}\mu: \mu\in \gbv{\beta}{\Oo}, \norm{\mu}_{\gbv{\beta}{\Oo}}\leq 1\}
 \]
 and let $r_{n}$ denote the infimum of the $r$ such that the set $
B_{n}$
may be covered by a finite number of balls of radius $r$ (measured in the $\norm{\cdot}_{\gbv{\beta}{\Oo}}$ norm). The formula of Nussbaum \cite{Nussbaum} states that 
  \begin{equation}\label{eq:nussbaum}
r_{ess}(\cL_{\xi,f})=\liminf_{n\to \infty}  \sqrt[n]{r_{n}}.
\end{equation}
By  Lemma~\ref{lem:compact}, we know that
 that $B_{0}$ is relatively compact in the $\TotalVar{\cdot}$ norm and
therefore, for each $\epsilon >0$, there exists a finite set 
$\{G_{i}\}_{i=1}^{N_{\epsilon}}$ of  subsets of $B_{0}$ whose union covers $B_{0}$ 
and such that
\begin{equation}\label{compact1}
\TotalVar{\smash{\mu-\tilde \mu}} \leq \epsilon \quad \text{  for all $\mu,\tilde \mu \in G_{i}$}.
\end{equation} Notice that \( r_{n} \) can be bounded above by the supremum of the diameters of the elements of any given cover of \( B_{n} \). 
Since the union of \( \{G_{i}\}_{i=1}^{N_{\epsilon}} \) is a cover of \( B_{0} \),  then \( \{\cL_{\xi,f}^{n}(G_{i})\}_{i=1}^{N_{\epsilon}} \) is a cover of \( B_{n} \) and therefore it is  sufficient to obtain an upper bound for the maximum diameter of the \( \cL_{\xi,f}^{n}(G_{i}) \). We use 
 the estimate on $\norm{\smash{\cL_{\xi,f}^{n}\mu}}_{\gbv{\beta}{\Oo}}$ from  Lemma~\ref{lem:LY2} and for convenience let 
 \[
 \lambda(n):=     10 \LL{\infty}{\xi^{(n)}}{\Oo}^{\beta} \LL{\infty}{\xi^{(n)}\cdot (f^{n})'}{\Oo}^{1-\beta}.
 \]
We therefore have  that for all $\mu,\tilde \mu \in G_{i}$ and $n\in \bN$ then
 \begin{equation*}
\norm{ \smash{ \cL_{\xi,f}^{n}\mu- \cL_{\xi,f}^{n}\tilde \mu}}_{\gbv{\beta}{\Oo}} \leq   \lambda({n}) \norm{\smash{\mu-\tilde \mu}}_{\gbv{\beta}{\Oo}} + C_{n} \TotalVar{\smash{\mu-\tilde \mu}} . 
\end{equation*}
Substituting \eqref{compact1} we have shown that  
$r_{n}\leq   \lambda({n})  + C_{n} \epsilon$.
We  choose  $\epsilon =\epsilon(n)$ small enough so that $r_{n} \leq 2 \lambda({n})$ and so \eqref{eq:nussbaum} implies the desired estimate on the essential spectral radius.
Now we know the estimate on the essential spectral radius the estimate $\TotalVar{\cL_{\xi,\map}} \leq \LL{\infty}{\xi \cdot f'}{\Oo}$ implies that the spectral radius is not greater than $\lambda_{2}$.

\end{document}